\documentclass{article}

\usepackage{amssymb, amsmath, amsthm, tikz, verbatim, graphicx,lmodern,indentfirst,enumerate,color}

\usepackage[titletoc,page]{appendix}
\usepackage[margin=1.4in]{geometry}

\newtheorem{thm}{Theorem}[section]
\newtheorem{lem}[thm]{Lemma}
\newtheorem{prop}[thm]{Proposition}

\newtheorem{oprob}[thm]{Open Problem}

\theoremstyle{definition}
\newtheorem{example}[thm]{Example}

\theoremstyle{definition}
\newtheorem{defn}[thm]{Definition}

\theoremstyle{remark}
\newtheorem{rem}[thm]{Remark}

\numberwithin{equation}{section}

\newcommand*\circled[1]{\tikz[baseline=(char.base)]{
            \node[shape=circle,draw,inner sep=0pt,minimum size=5mm] (char) {#1};}}

\newcommand{\rootsE}[7]{\circled{#1}
            \begin{tabular}{ccccc}
             &&#2&& \\
             #3&#4&#5&#6&#7
             \end{tabular}}

\makeatletter
\newcommand{\rmnum}[1]{\romannumeral #1}
\newcommand{\Rmnum}[1]{\expandafter\@slowromancap\romannumeral #1@}
\makeatother
\begin{document}

\title{Complete reducibility of subgroups of reductive algebraic groups over nonperfect fields \Rmnum{1}}
\author{Tomohiro Uchiyama\\
Department of Mathematics, University of Auckland, \\
Private Bag 92019, Auckland 1142, New Zealand\\
\texttt{email:tuch540@aucklanduni.ac.nz}}
\date{}
\maketitle 

\begin{abstract}
Let $k$ be a nonperfect field of characteristic $2$. Let $G$ be a $k$-split simple algebraic group of type $E_6$ (or $G_2$) defined over $k$. In this paper, we present the first examples of nonabelian non-$G$-completely reducible $k$-subgroups of $G$ which are $G$-completely reducible over $k$. Our construction is based on that of subgroups of $G$ acting non-separably on the unipotent radical of a proper parabolic subgroup of $G$ in our previous work. We also present examples with the same property for a non-connected reductive group $G$. Along the way, several general results concerning complete reducibility over nonperfect fields are proved using the recently proved Tits center conjecture for spherical buildings. In particular, we show that under mild conditions a $k$-subgroup of $G$ is pseudo-reductive if it is $G$-completely reducible over $k$.   
\end{abstract}

\section{Introduction} 
Let $k$ be an arbitrary field. We write $\bar k$ for an algebraic closure of $k$. Let $G/k$ be a connected reductive algebraic group defined over $k$: we regard $G$ as a ${\bar k}$-defined algebraic group together with a choice of $k$-structure~\cite[AG.11]{Borel-AG-book}. Following Serre~\cite{Serre-building},  define:
\begin{defn}
A closed subgroup $H$ of $G$ is \emph{$G$-completely reducible over $k$} (\emph{$G$-cr over $k$} for short) if whenever $H$ is contained in a $k$-parabolic subgroup $P$ of $G$, $H$ is contained in some $k$-Levi subgroup $L$ of $P$. In particular, if $H$ is not contained in any $k$-parabolic subgroup of $G$, $H$ is \emph{$G$-irreducible over $k$} (\emph{$G$-ir over $k$} for short). Note that we do not require $H$ to be $k$-defined. 
\end{defn}
Our definition is a slight generalization of Serre's original definition in~\cite{Serre-building}, where $H$ is assumed to be $k$-defined. This generalized definition was used in~\cite{Bate-cocharacterbuildings-Arx} and~\cite{Bate-cocharacter-Arx}.

The notion of $G$-complete reducibility over $k$ is a natural generalization of that of complete reducibility in representation theory: if $G=GL(V)$ for some finite dimensional $k$-vector space $V$, a subgroup $H$ of $G$ acts on $V$ semisimply over $k$ if and only if $H$ is $G$-complete reducible over $k$~\cite[Sec.~1.3]{Serre-building}. We say that a subgroup $H$ of $G$ is $G$-cr ($G$-ir) if $H$ is $G$-cr over $\bar k$ ($G$-ir over $\bar k$) regarding $G$ to be defined over $\bar k$. By a subgroup $H$ of $G$, we always mean a closed subgroup of $G$ unless otherwise stated. 

Complete reducible subgroups are much studied, but most studies so far considered complete reducibility over $\bar k$ only; see~\cite{Bate-geometric-Inventione},~\cite{Liebeck-Seitz-memoir},~\cite{Stewart-nonGcr},~\cite{Thomas-irreducible-JOA}. Not much is known about completely reducible subgroups over arbitrary $k$ except for a few results and important examples in~\cite{Bate-cocharacterbuildings-Arx},~\cite{Bate-cocharacter-Arx},~\cite[Sec.~6.5]{Bate-geometric-Inventione},~\cite{Bate-separable-Paris},~\cite[Sec.~7]{Bate-separability-TransAMS},~\cite[Thm.~1.8]{Uchiyama-Classification-pre},~\cite[Sec.~4]{Uchiyama-Separability-JAlgebra}. We write $k_s$ for a separable closure of $k$. The main result of this paper is the following:

\begin{thm}\label{G-cr over k non G-cr}
Let $k=k_s$ be a nonperfect field of characteristic $2$. Let $G/k$ be a simple algebraic group of type $E_6$ (or $G_2$). Then there exists a nonabelian $k$-subgroup $H$ of $G$ such that $H$ is $G$-cr over $k$, but not $G$-cr.
\end{thm}

Several examples of an abelian subgroup $H<G$ such that $H$ is $G$-cr over $k$ but not $G$-cr are known; see Example~\ref{PGL-nonplongeable},~\cite{Gille-E8unipotent-QJM},~\cite{Tits-unipotent-Utrecht},~\cite{Tits-Note1-CollegeFrance}. Note that in these examples, $H<G$ is generated by a \emph{$k$-anisotropic unipotent element}~\cite{Tits-unipotent-Utrecht}. 

\begin{defn}
Let $G/k$ be a reductive algebraic group. A unipotent element $u$ of $G$ is \emph{$k$-nonplongeable unipotent} if $u$ is not contained in the $k$-unipotent radical of any $k$-parabolic subgroup of $G$. In particular, if $u$ is not contained in any $k$-parabolic subgroup of $G$, $u$ is \emph{$k$-anisotropic unipotent}. 
\end{defn}

By the $k$-unipotent radical of an affine $k$-group $N$, we mean the maximal connected unipotent normal $k$-subgroup of $N$. It is clear that a subgroup $H$ of $G$ generated by a $k$-anisotropic unipotent element is $G$-ir over $k$. Since $H$ is unipotent, the classical result of Borel-Tits~\cite[Prop.~3.1]{Borel-Tits-unipotent-invent} shows that $H$ is not $G$-cr; see Example~\ref{PGL-nonplongeable}. 

The next result~\cite[Thm.~1.1]{Bate-separable-Paris} shows that the nonperfectness assumption of $k$ in Theorem~\ref{G-cr over k non G-cr} is necessary. Recall that if $k$ is perfect, we have $k_s=\bar k$.
\begin{prop}\label{separable}
Let $k$ be an arbitrary field. Let $G/k$ be a connected reductive algebraic group. Then a subgroup $H$ of $G$ is $G$-cr over $k$ if and only if $H$ is $G$-cr over $k_s$.
\end{prop}

The forward direction of Proposition~\ref{separable} holds for a non-connected reductive group $G$ in an appropriate sense (see Definition~\ref{non-connected-G-cr}). The reverse direction depends on the \emph{Tits center conjecture}  (Theorem~\ref{TCC}), but this method does not work for non-connected $G$; see~\cite{Uchiyama-Nonperfectopenproblem-pre}. 

In Section 3, we present an example of a subgroup $H$ for $G=E_6$ (or $G_2$) satisfying the properties of Theorem~\ref{G-cr over k non G-cr}. The key to our construction is the notion of a \emph{non-separable action}~\cite[Def.~1.5]{Uchiyama-Separability-JAlgebra}.

\begin{defn}
Let $H$ and $N$ be affine algebraic groups. Suppose that $H$ acts on $N$ by group automorphisms. The action of $H$ is called \emph{separable} in $N$ if $\text{Lie}\;C_N(H)=\mathfrak{c}_{\textup{Lie} N} (H)$. Note that the condition means that the scheme-theoretic centralizer of $H$ in $N$ (in the sense of~\cite[Def.~A.1.9]{Conrad-pred-book}) is smooth. 
\end{defn}

Note that the notion of a separable action is a slight generalization of that of a separable subgroup~\cite[Def.~1.1]{Bate-separability-TransAMS}. See~\cite{Bate-separability-TransAMS} and~\cite{Herpel-smoothcentralizerl-trans} for more on separability. It is known that if the characteristic $p$ of $k$ is very good for $G$, every subgroup of $G$ is separable~\cite[Thm.~1.2]{Bate-separability-TransAMS}. This suggests that we need to work in small $p$. Proper non-separable subgroups are hard to find. Only a handful of such examples are known~\cite[Sec.~7]{Bate-separability-TransAMS},~\cite{Uchiyama-Classification-pre},\cite{Uchiyama-Separability-JAlgebra}. 

\begin{rem}
The examples of subgroups $H$ of $G$ in Section~3 are $G$-ir over $k$ but not $G$-cr. So, we can regard these examples as a generalization of $k$-anisotropic unipotent elements. 
\end{rem}

Next, we consider a non-connected case. Again, non-separability is the key to our construction, but the computations are much simpler than in the connected cases.
\begin{thm}\label{non-connected G-cr over k non-G-cr}
Let $k=k_s$ be a nonperfect field of characteristic $2$. Let $\tilde G/k$ be a simple algebraic group of type $A_4$. Let $G:=\tilde G\rtimes \langle \sigma \rangle$ where $\sigma$ is the non-trivial graph automorphism of $\tilde G$. Then there exists a nonabelian $k$-subgroup $H$ of $G$ such that $H$ is $G$-cr over $k$, but not $G$-cr.
\end{thm}

In Section 2 we extend several existing results concerning complete reducibility over $\bar k$ to a nonperfect $k$. Most arguments are based on~\cite{Bate-geometric-Inventione} and the Tits center conjecture (Theorem~\ref{TCC}) in spherical buildings. We also consider the relationship between complete reducibility over $k$ and \emph{pseudo-reductivity}~\cite{Conrad-pred-book}. Recall:
\begin{defn}
Let $k$ be a field. Let $G/k$ be a smooth connected algebraic group. If the $k$-unipotent radical $R_{u,k}(G)$ of $G$ is trivial, $G$ is called \emph{pseudo-reductive}.
\end{defn}
Note that if $k$ is perfect, pseudo-reductive groups are reductive. Our main result on pseudo-reductivity is the following:

\begin{thm}\label{main}
Let $k=k_s$ be a field. Let $G/k$ be a semisimple simply connected algebraic group. Assume that $[k:k^p]\leq p$. If a $k$-subgroup $H$ of $G$ is $G$-cr over $k$, then $H$ is pseudo-reductive.
\end{thm}

Let $G/k$ be connected reductive. A standard argument~\cite[Ex.~3.2.2(a)]{Serre-building} (which depends on~\cite[Prop.~3.1]{Borel-Tits-unipotent-invent}) shows that a $G$-cr subgroup of $G$ is reductive, hence pseudo-reductive. However when $k$ is nonperfect we have:
\begin{prop}\label{pseudoreductive}
Let $k$ be a nonperfect field of characteristic $2$. Let $G=PGL_2$. Then there exists a $k$-subgroup $H$ of $G$ such that $H$ is $G$-cr over $k$, but not pseudo-reductive.
\end{prop}

We extend~\cite[Lem.~2.12]{Bate-geometric-Inventione} using the notion of a \emph{central isogeny}. Recall~\cite[Sec.~22.3]{Borel-AG-book}:
\begin{defn}
Let $k$ be a field. Let $G_1/k$ and $G_2/k$ be connected reductive. A $k$-isogeny $f:G_1\rightarrow G_2$ is \emph{central} if 
$\textup{ker}\;df_1$ is central in $\mathfrak{g_1}$ where $df_1$ is the differential of $f$ at the identity of $G_1$.
\end{defn} 

\begin{prop}\label{isogeny}
Let $k$ be a field. Let $G_1/k$ and $G_2/k$ be connected reductive. Let $H_1$ and $H_2$ be (not necessarily $k$-defined) subgroups of $G_1$ and $G_2$ respectively. Let $f:G_1\rightarrow G_2$ be a central $k$-isogeny. 
\begin{enumerate}
\item If $H_1$ is $G_1$-cr over $k$, then $f(H_1)$ is $G_2$-cr over $k$.
\item If $H_2$ is $G_2$-cr over $k$, then $f^{-1}(H_2)$ is $G_1$-cr over $k$.
\end{enumerate}
\end{prop}

Many problems concerning complete reducibility over nonperfect fields are still open. For example:
\begin{oprob}\label{centralizerquestion}
Let $k$ be a field. Let $G/k$ be connected reductive. Suppose that a $k$-subgroup $H$ of $G$ is $G$-cr over $k$. Is the centralizer $C_G(H)$ of $H$ in $G$ $G$-cr over $k$?
\end{oprob}  

It is known that if $k=\bar k$, the answer to Open Problem~\ref{centralizerquestion} is yes; see~\cite[Cor.~3.17]{Bate-geometric-Inventione}. Also, we show
\begin{prop}\label{TCC-centralizer}
Let $k=k_s$. Let $H/k$ be a subgroup of $G/k$. Suppose that $H$ is $G$-ir over $k$. Then $C_G(H)$ is $G$-cr over $k$. 
\end{prop}
See~\cite{Uchiyama-Nonperfectopenproblem-pre} for more on this problem and other related open problems.

Here is the structure of the paper. In Section~2, we set out the notation. Then, in Section~3, we prove various general results including Theorem~\ref{main}, Proposition~\ref{pseudoreductive}, Proposition~\ref{isogeny}, and Proposition~\ref{TCC-centralizer}. In Section~4, we prove Theorem~\ref{G-cr over k non G-cr}. Then, in Section~5, we consider non-connected $G$, and prove Theorem~\ref{non-connected G-cr over k non-G-cr}. Finally, in Section~6, we consider further applications of non-separable actions for non-connected $G$, and prove Theorem~\ref{non-connected G-cr non-G-cr over k} and Theorem~\ref{G-crM-cr}.

\section{Preliminaries}
Throughout, we denote by $k$ a separably closed field. Although some results hold for an arbitrary field, our assumption on $k$ makes the exposition cleaner. Our references for algebraic groups are~\cite{Borel-AG-book},~\cite{Borel-Tits-Groupes-reductifs},~\cite{Borel-Tits-Groupes-reductifs-complements},~\cite{Humphreys-book1}, and~\cite{Springer-book}. 

Let $G/k$ be a (possibly non-connected) affine algebraic group defined over $k$. By a $k$-group $G$, we mean a $\bar k$-defined affine algebraic group with a $k$-structure~\cite[AG.11]{Borel-AG-book}. We write $G(k)$ for the set of $k$-points of $G$. The \emph{unipotent radical} of $G$ is denoted by $R_u(G)$, and $G$ is called \emph{reductive} if $R_u(G)=\{1\}$. A reductive group $G$ is called \emph{simple} as an algebraic group if $G$ is connected and all proper normal subgroups of $G$ are finite. We write $X_k(G)$ and $Y_k(G)$ for the set of $k$-characters and $k$-cocharacters of $G$ respectively. 

Let $G/k$ be reductive. Fix a $k$-split maximal torus $T$ of $G$ (such a $T$ exists by~\cite[Cor.~18.8]{Borel-AG-book}). Let $\Psi(G,T)$ denote the set of roots of $G$ with respect to $T$. We sometimes write $\Psi(G)$ for $\Psi(G,T)$. Let $\zeta\in\Psi(G)$. We write $U_\zeta$ for the corresponding root subgroup of $G$. We define $G_\zeta := \langle U_\zeta, U_{-\zeta} \rangle$. Let $\zeta, \xi \in \Psi(G)$. Let $\xi^{\vee}$ be the coroot corresponding to $\xi$. Then $\zeta\circ\xi^{\vee}:\bar k^{*}\rightarrow \bar k^{*}$ is a $k$-homomorphism such that $(\zeta\circ\xi^{\vee})(a) = a^n$ for some $n\in\mathbb{Z}$.
Let $s_\xi$ denote the reflection corresponding to $\xi$ in the Weyl group of $G$. Each $s_\xi$ acts on the set of roots $\Psi(G)$ by the following formula~\cite[Lem.~7.1.8]{Springer-book}:
$
s_\xi\cdot\zeta = \zeta - \langle \zeta, \xi^{\vee} \rangle \xi. 
$
\noindent By \cite[Prop.~6.4.2, Lem.~7.2.1]{Carter-simple-book} we can choose $k$-homomorphisms $\epsilon_\zeta : \bar k \rightarrow U_\zeta$  so that 
$
n_\xi \epsilon_\zeta(a) n_\xi^{-1}= \epsilon_{s_\xi\cdot\zeta}(\pm a)
            \text{ where } n_\xi = \epsilon_\xi(1)\epsilon_{-\xi}(-1)\epsilon_{\xi}(1).  \label{n-action on group}
$

We recall the notions of $R$-parabolic subgroups and $R$-Levi subgroups from~\cite[Sec.~2.1--2.3]{Richardson-conjugacy-Duke}. These notions are essential to define $G$-complete reducibility for subgroups of non-connected reductive groups; see~\cite{Bate-nonconnected-PAMS} and~\cite[Sec.~6]{Bate-geometric-Inventione}. 

\begin{defn}
Let $X/k$ be a $k$-affine variety. Let $\phi : \bar k^*\rightarrow X$ be a $k$-morphism of $k$-affine varieties. We say that $\displaystyle\lim_{a\rightarrow 0}\phi(a)$ exists if there exists a $k$-morphism $\hat\phi:\bar k\rightarrow X$ (necessarily unique) whose restriction to $\bar k^{*}$ is $\phi$. If this limit exists, we set $\displaystyle\lim_{a\rightarrow 0}\phi(a) = \hat\phi(0)$.
\end{defn}

\begin{defn}
Let $\lambda\in Y_k(G)$. Define
$
P_\lambda := \{ g\in G \mid \displaystyle\lim_{a\rightarrow 0}\lambda(a)g\lambda(a)^{-1} \text{ exists}\}, $\\
$L_\lambda := \{ g\in G \mid \displaystyle\lim_{a\rightarrow 0}\lambda(a)g\lambda(a)^{-1} = g\}, \,
R_u(P_\lambda) := \{ g\in G \mid  \displaystyle\lim_{a\rightarrow0}\lambda(a)g\lambda(a)^{-1} = 1\}. 
$
\end{defn}
We call $P_\lambda$ an $R$-parabolic subgroup of $G$, $L_\lambda$ an $R$-Levi subgroup of $P_\lambda$. Note that $R_u(P_\lambda)$ a unipotent radical of $P_\lambda$. If $\lambda$ is $k$-defined, $P_\lambda$, $L_\lambda$, and $R_u(P_\lambda)$ are $k$-defined~\cite[Sec.~2.1-2.3]{Richardson-conjugacy-Duke}. If $G$ is connected, $R$-parabolic subgroups and $R$-Levi subgroups are parabolic subgroups and Levi subgroups in the usual sense~\cite[Prop.~8.4.5]{Springer-book}. It is well known that $L_\lambda = C_G(\lambda(\bar k^*))$. 

Let $M/k$ be a reductive subgroup of $G$. Then, there is a natural inclusion $Y_k(M)\subseteq Y_k(G)$ of $k$-cocharacter groups. Let $\lambda\in Y_k(M)$. We write $P_\lambda(G)$ or just $P_\lambda$ for the $k$-parabolic subgroup of $G$ corresponding to $\lambda$, and $P_\lambda(M)$ for the $k$-parabolic subgroup of $M$ corresponding to $\lambda$. It is clear that $P_\lambda(M) = P_\lambda(G)\cap M$ and $R_u(P_\lambda(M)) = R_u(P_\lambda(G))\cap M$. Now we define:
\begin{defn}\label{non-connected-G-cr}
Let $G/k$ be a (possibly non-connected) reductive algebraic group. A subgroup $H$ of $G$ is $G$-cr over $k$ if whenever $H$ is contained in a $k$-defined $R$-parabolic subgroup $P_\lambda$, $H$ is contained in a $k$-defined $R$-Levi subgroup of $P_\lambda$.
\end{defn}

\section{General results}
\subsection{The Tits center conjecture}
Let $G/k$ be connected reductive. We write $\Delta(G)$ for the Tits spherical building of $G$~\cite{Tits-book}. Recall that each simplex in $\Delta(G)$ corresponds to a proper $k$-parabolic subgroup of $G$, and the conjugation action of $G(k)$ on itself induces building automorphisms of $\Delta(G)$. The following is the so-called \emph{Tits center conjecture} (\cite[Sec.~2.4]{Serre-building} and~\cite[Lem.~1.2]{Tits-colloq}), which was recently proved by Tits, M\"{u}hlherr, Leeb, and Ramos-Cuevas~\cite{Leeb-Ramos-TCC-GFA},~\cite{Muhlherr-Tits-TCC-JAlgebra},~\cite{Ramos-centerconj-Geo}:

\begin{thm}\label{TCC}
Let $X$ be a convex contractible subcomplex of $\Delta(G)$. Then there exists a simplex in $X$ that is stabilized by all automorphisms of $\Delta(G)$ stabilizing $X$. 
\end{thm}

In~\cite[Def.~2.2.1]{Serre-building} Serre defined that a convex subcomplex $X$ of $\Delta(G)$ is $\Delta(G)$-\emph{completely reducible over $k$} ($\Delta(G)$-cr over $k$ for short) if for every simplex $x\in X$, there exists a simplex $x'\in X$ opposite to $x$ in $X$. Serre showed~\cite[Thm.~2]{Serre-building}:
\begin{prop}\label{TCC-contractible}
Let $X$ be a convex subcomplex of $\Delta(G)$. Then $X$ is $\Delta(G)$-cr over $k$ if and only if $X$ is not contractible.
\end{prop}

Combining Theorem~\ref{TCC} with Proposition~\ref{TCC-contractible}, and translating the result into the language of algebraic groups we obtain

\begin{prop}\label{TCC-group}
Let $H$ be a (not necessarily $k$-defined) subgroup of $G/k$. If $H$ is not $G$-cr over $k$, then there is a proper $k$-parabolic subgroup $P$ of $G$ such that $P$ contains $H$ and $N_G(H)(k)$ where $N_G(H)(k):=G(k)\cap N_G(H)$.
\end{prop}

\begin{proof}
Let $\Delta(G)^{H}$ be the fixed point subcomplex of the action of $H$. Then the set of $\Delta(G)^{H}$ is a convex subcomplex of $\Delta(G)$ by~\cite[Prop.~3.1]{Serre-building} and $\Delta(G)^H$ corresponds to the set of all proper $k$-parabolic subgroups of $G$ containing $H$. Since $H$ is not $G$-cr over $k$, there exists a proper $k$-parabolic subgroup of $G$ containing $H$ such that $H$ is not contained in any opposite of $P$. So, $\Delta(G)^H$ is contractible by Proposition~\ref{TCC-contractible}. It is clear that $N_G(H)(k)$ induces automorphisms of $\Delta(G)$ stabilizing $\Delta(G)^H$. By Theorem~\ref{TCC}, there exists a simplex $s_P$ in $\Delta(G)^H$ stabilized by automorphisms induced by $N_G(H)(k)$. Since parabolic subgroups are self-normalizing, we have $N_G(H)(k) < P$. 
\end{proof}

Note that under the assumption of Proposition~\ref{TCC-group}, $N_G(H)$ is not necessarily $k$-defined even when $H$ is $k$-defined. So, we might not have a proper $k$-parabolic subgroup containing $H$ and $N_G(H)$ .

\begin{proof}[Proof of Proposition~\ref{TCC-centralizer}]
Suppose that $C_G(H)$ is not $G$-cr over $k$. Since $H$ normalizes $C_G(H)$, by Proposition~\ref{TCC-group}, there exists a proper $k$-parabolic subgroup of $P$ of $G$ containing $H(k)$. Since $k=k_s$, $H(k)$ is dense in $H$ by~\cite[AG.13.3]{Borel-AG-book}. So $H\leq P$. This is a contradiction since $H$ is $G$-ir over $k$. 
\end{proof}

\subsection{Complete reducibility and pseudo-reductivity}
The main task in this section is to prove Theorem~\ref{main}. Before that, we need some preparations:

\begin{lem}\label{G-cr-L-cr}
Let $G/k$ be connected reductive, and let $H$ be a (not necessarily $k$-defined) subgroup of $G$. Let $L$ be a $k$-Levi subgroup of $G$ containing $H$. Then $H$ is $G$-cr over $k$ if and only if $H$ is $L$-cr over $k$.
\end{lem}
\begin{proof}
This is \cite[Thm.~1.4]{Bate-cocharacterbuildings-Arx}.
\end{proof}

The next result is a slight generalization of~\cite[Prop.~2.9]{Serre-building}, where Serre assumed the subgroup $N$ is $k$-defined. Note that Serre's argument assumed that Theorem~\ref{TCC} holds, but this was not known at the time. We have translated Serre's building-theoretic argument into a group-theoretic one.
\begin{prop}\label{Normal}
Let $G/k$ be connected reductive. Let $H/k$ be a subgroup of $G$ such that $H$ is $G$-cr over $k$. If $N$ is a (not necessarily $k$-defined) normal subgroup of $H$, then $N$ is $G$-cr over $k$.
\end{prop}
\begin{proof}
Let $P$ be a minimal $k$-parabolic subgroup of $G$ containing $H$. Since $H$ is $G$-cr over $k$, there exists a $k$-Levi subgroup $L$ of $P$ containing $H$. If $N$ is $L$-cr over $k$, by Lemma~\ref{G-cr-L-cr}, we are done. So suppose that $N$ is not $L$-cr over $k$. Let $\Delta(L)$ be the spherical building corresponding to $L$. Let $\Delta(L)^N$ be the fixed point subcomplex of $\Delta(L)$. Since $N\trianglelefteq H\leq L$ and $N$ is not $L$-cr over $k$, by Proposition~\ref{TCC-group}, there exists a proper $k$-parabolic subgroup $P_L$ of $L$ containing $N$ and $H(k)$. Since $k=k_s$, $H(k)$ is dense in $H$ by~\cite[AG.13.3]{Borel-AG-book}. So $H\leq P_L < L$. Then $H\leq P_c\ltimes R_u(P) < L\ltimes R_u(P) = P$. Since $P_c\ltimes R_u(P)$ is a $k$-parabolic subgroup of $G$ by~\cite[Sec.~4.4(c)]{Borel-Tits-Groupes-reductifs}, this is a contradiction by the minimality of $P$.
\end{proof} 

We also need the following deep result which was conjectured by Tits~\cite{Tits-Note1-CollegeFrance} and proved by Gille~\cite{Gille-unipotent-Duke}.
\begin{prop}\label{plongeable}
Let $G/k$ be a semisimple simply connected algebraic group. If $[k:k^p]\leq p$, then every unipotent subgroup of $G(k)$ is $k$-plongeable.
\end{prop}

Now we are ready: 
\begin{proof}[Proof of Theorem~\ref{main}]
Since $R_{u,k}(H)(k)$ is a unipotent subgroup of $G(k)$, by Proposition~\ref{plongeable}, there exists a $k$-parabolic subgroup $P$ of $G$ such that $R_{u,k}(H)(k)\leq R_u(P)$. Since $R_{u,k}(H)(k)$ is a normal subgroup of $H$, and $H$ is $G$-cr over $k$, $R_{u,k}(H)(k)$ is $G$-cr over $k$ by Proposition~\ref{Normal}. So $R_{u,k}(H)(k)$ is contained in some $k$-Levi subgroup of $P$. Thus $R_{u,k}(H)(k)=1$. Since $R_{u,k}(H)(k)$ is dense in $R_{u,k}(H)$, we are done. 
\end{proof}

Note that in Proposition~\ref{plongeable}, the condition $[k:k^p]\leq p$ was necessary since Tits showed the following~\cite[Thm.~7]{Tits-Note3-CollegeFrance}.

\begin{prop}\label{semisimple-bad}
Let $G/k$ be a simple simply connected algebraic group. If $[k:k^p]\geq p^2$ and $p$ is bad for $G$, then $G(k)$ has a $k$-nonplongeable unipotent element.
\end{prop}

We quickly review an example of abelian $H<G$ such that $H$ is $G$-cr over $k$ but not $G$-cr. Although this example is known, it has not been interpreted in the context of $G$-complete reducibility. 

\begin{example}\label{PGL-nonplongeable}
Let $k$ be a nonperfect field of characteristic $p=2$. Let $a\in k\backslash k^2$. Let $G/k=PGL_2$. We write $\bar A$ for the image in $PGL_2$ of $A\in GL_2$. Set $u= \overline{\left[\begin{array}{cc}
                        0 & a \\
                        1 & 0 \\
                         \end{array}  
                        \right]}\in G(k)$. Let $U:=\langle u \rangle$. Then $U$ is unipotent, so by the classical result of Borel-Tits~\cite[Prop.~3.1]{Borel-Tits-unipotent-invent} $U$ is contained in the unipotent radical of a proper parabolic subgroup of $G$. So $U$ is not $G$-cr. However $U$ is not contained in any proper $k$-parabolic subgroup of $G$ since there is no nontrivial $k$-defined flag of $\mathbb{P}^1_k$ stabilized by $U$. So $U$ is $G$-ir over $k$, hence $G$-cr over $k$. Note that this example shows that~\cite[Prop.~3.1]{Borel-Tits-unipotent-invent} fails over a nonperfect $k$.
\end{example}

\begin{proof}[Proof of Proposition~\ref{pseudoreductive}]
Let $k$ be a nonperfect field of characteristic $2$. Let $a\in k\backslash k^2$. Let $G=PGL_2$ and $H:=\left\{\overline{\left[\begin{array}{cc}
                                                      x & ay \\
                                                      y & x \\
                                                    \end{array}  
                                             \right]} \in PGL_2(\bar k) \mid x, y \in \bar k\right\}$. Then $H$ is a connected $k$-defined unipotent subgroup of $G$. Therefore $H$ is not pseudo-reductive. It is clear that $H$ contains a $k$-anisotropic unipotent element $\overline{\left[\begin{array}{cc}
                                                      0 & a \\
                                                      1 & 0 \\
                                                    \end{array}  
                                             \right]}$ of $G$. So $H$ is $G$-ir over $k$.    
\end{proof}
\begin{rem}
Let $k$, $a$, $G$, $H$ be as in the proof of Proposition~\ref{pseudoreductive}. Note that the subgroup $H$ is the centralizer of the subgroup $U:=\left\langle \overline{\left[\begin{array}{cc}
                                                      0 & a \\
                                                      1 & 0 \\
                                                    \end{array}  
                                             \right]} \right\rangle$ of $G$. So without the perfectness assumption of $k$ we have a counterexample to~\cite[Prop.~3.12]{Bate-geometric-Inventione} which states that the centralizer of a $G$-cr over $k$ subgroup is reductive. Reducitivity of the centralizer was a key ingredient in the proof of~\cite[Cor.~3.17]{Bate-geometric-Inventione}. Although our example does not give a negative answer to Open Problem~\ref{centralizerquestion}, it suggests that the answer is no.  
\end{rem}

\subsection{Complete reducibility under isogenies}

\begin{proof}[Proof of Proposition~\ref{isogeny}]
Suppose that $f(H_1)$ is contained in a $k$-parabolic subgroup $P_{\mu}$ of $G_2$ where $\mu\in Y_k(G_2)$. Then $H_1<f^{-1}(P_{\mu})=P_{\lambda}$ for some $\lambda\in Y_k(G_1)$ since $f^{-1}(P_{\mu})$ is a $k$-defined parabolic subgroup of $G_1$ by~\cite[Thm.~22.6]{Borel-AG-book}. So $P_{\mu}=f(P_{\lambda})=P_{f\circ\lambda}$ by~\cite[Lem.~2.11]{Bate-geometric-Inventione}. 
Since $H_1$ is $G_1$-cr over $k$, there exists a $k$-Levi subgroup $L$ of $P_{\lambda}$ containing $H_1$. We can set $L:=u\cdot L_{\lambda}$ for some $u\in R_u(P_\lambda)(k)$ since $k$-Levi subgroups of $P_{\lambda}$ are $R_u(P_{\lambda})(k)$-conjugate by~\cite[Prop.~20.5]{Borel-AG-book}. Then $f(H_1)<f(u)\cdot f(L_\lambda)
=f(u)\cdot L_{f\circ \lambda}$ by~\cite[Lem.~2.11]{Bate-geometric-Inventione}. Since $f\circ \lambda$ is a $k$-cocharacter of $G_2$ and $f(u)$ is a $k$-point of $f(R_u(P_\lambda))=R_u(P_{f\circ \lambda})$ (\cite[Lem.~2.11]{Bate-geometric-Inventione}), $f(u)\cdot L_{f\circ \lambda}$ is a $k$-Levi subgroup of $P_{f\circ\lambda}=P_{\mu}$ containing $f(H_1)$. So we have the first part of the proposition.

Now, suppose that there exists a $k$-parabolic subgroup $P_{\lambda'}$ of $G_1$ containing $f^{-1}(H_2)$ where $\lambda'\in Y_k(G_1)$. Then there exists some $\mu'\in Y_k(G_2)$ such that $P_{\lambda'}=f^{-1}(P_{\mu'})$ since every $k$-parabolic subgroup of $G_1$ is the inverse image of a $k$-parabolic subgroup of $G_2$ by~\cite[Thm.~22.6]{Borel-AG-book}. So $H_2<P_{\mu'}$. Since $H_2$ is $G_2$-cr over $k$, there exists a $k$-Levi subgroup $L'$ of $P_{\mu'}$ containing $H_2$. By the same argument as in the last paragraph, set $L':=u'\cdot L_{\mu'}=L_{u'\cdot \mu'}$ for some $u'\in R_u(P_{\mu'})(k)$. Then $f^{-1}(H_2)<f^{-1}(L_{u'\cdot \mu'})<f^{-1}(P_{u'\cdot \mu'})=f^{-1}(P_{\mu'})=P_{\lambda'}$. Note that $f^{-1}(L_{u'\cdot \mu'})$ is a Levi subgroup of $f^{-1}(P_{u'\cdot \mu'})=P_{\lambda'}$ by~\cite[Lem.~2.11]{Bate-geometric-Inventione}, and it is $k$-defined by~\cite[Cor.~22.5]{Borel-AG-book} since $L_{u'\cdot \mu'}$ is a $k$-defined subgroup of $G_2$ containing a maximal torus of $G_2$. We are done.
\end{proof}

Note that if $k=\bar{k}$, Proposition~\ref{isogeny} holds without assuming $f$ central, but if $k$ is nonperfect, the next example shows that the first part of Proposition~\ref{isogeny} does not necessarily hold:
\begin{example}\label{counterexample}
Let $k$ be a nonperfect field of characteristic $2$. Let $a\in k\backslash k^2$. Let $G_1=G_2=PGL_2$, and $f$ be the Frobenius map. Let $h_1=
\small \overline{\left[\begin{array}{ll}
         0 & a \\
         1 & 0 \\
         \end{array} 
   \right]}$. \normalsize Then it is clear that $H_1:=\langle h_1 \rangle$ is $G_1$-ir over $k$, but $H_2:=\langle f(h_1) \rangle=
\small \left\langle\overline{\left[\begin{array}{ll}
         0 & a^2 \\
         1 & 0 \\
         \end{array} 
    \right]}\right\rangle$ \normalsize is not $G_2$-cr over $k$; $H_2$ acts on $\mathbb{P}^1_k$ with a $k$-defined $H_2$-invariant subspace spanned by $[a,1]$ which has no $k$-defined $H_2$-invariant complementary subspace. 
\end{example}
\begin{rem}
Let $f:SL_2\rightarrow PGL_2$ be the canonical projection. Take the same $H_1$ as in Example~\ref{counterexample}. Then $f^{-1}(H_1)=\left\langle \small \left[\begin{array}{ll}
         0 & \sqrt{a} \\
         \sqrt{a}^{-1} & 0 \\
         \end{array} 
   \right]\right\rangle$ is not $k$-defined, but $f^{-1}(H_1)$ is $G$-ir over $k$. 
\end{rem}

\begin{oprob}
Does the second part of Proposition~\ref{isogeny} hold without assuming $f$ central?
\end{oprob}

\section{Proof of Theorem~\ref{G-cr over k non G-cr}}
For the rest of the paper, we assume $k=k_s$ is a nonperfect field of characteristic $2$ and $a\in k\backslash k^2$. 

\subsection{The $G_2$ example}
Let $G/k$ be a simple algebraic group of type $G_2$. Fix a $k$-split maximal torus $T$ of $G$ and a $k$-Borel subgroup of $G$ containing $T$. Let $\Sigma=\{ \alpha, \beta\}$ be the set of simple roots corresponding to $B$ and $T$ where $\alpha$ is short and $\beta$ is long. Then the set of roots of $G$ is $\Psi=\{\pm \alpha, \pm \beta, \pm (\alpha+\beta), \pm (2\alpha+\beta), \pm(3\alpha+\beta), \pm(3\alpha+2\beta)\}$. Let $b\in k^*$ such that $b^3=1$ and $b\neq 1$. Let $n_\alpha:=\epsilon_{-\alpha}(1)\epsilon_\alpha(1)\epsilon_{-\alpha}(1)$, and $t:=\alpha^{\vee}(b)$. Let $L_\alpha:=\langle T, G_\alpha \rangle$ and $P_\alpha:=\langle L_\alpha, U_\beta, U_{\alpha+\beta}, U_{2\alpha+\beta}, U_{3\alpha+\beta}, U_{3\alpha+2\beta}\rangle=P_{(3\alpha+2\beta)^{\vee}}$.  

In the following computation, we use the commutation relations for root subgroups of $G$; see~\cite[Sec.~33.5]{Humphreys-book1}. Define 
\begin{equation*}
K:=\langle n_\alpha \rangle, \; v(\sqrt a):=\epsilon_{-\beta}(\sqrt a)\epsilon_{-3\alpha-\beta}(\sqrt a),\; M:=\langle n_\alpha, t \rangle.
\end{equation*}

Let
\begin{equation*}
H:=\langle v(\sqrt a)\cdot M, \; \epsilon_{2\alpha+\beta}(1)\rangle =\langle n_\alpha \epsilon_{-3\alpha-2\beta}(a), t, \epsilon_{2\alpha+\beta}(1) \rangle.
\end{equation*}

\begin{prop}\label{G2ir}
$H$ is $G$-ir over $k$.
\end{prop}
\begin{proof}
Let
\begin{alignat*}{2}
\widetilde H:&=v(\sqrt a)^{-1}\cdot H = \langle n_\alpha, t, \epsilon_{2\alpha+\beta}(1)\epsilon_{3\alpha+\beta}(a)\epsilon_{3\alpha+2\beta}(\sqrt a)\epsilon_{\alpha+\beta}(\sqrt a)\epsilon_{-\alpha}(\sqrt a)\rangle. 
\end{alignat*}

It is clear that $L_\alpha$ is a Levi subgroup of $G$ containing $M$. Since $M$ is not contained in any Borel subgroup of $L_\alpha$, $M$ is $L$-ir. So $M$ is $G$-cr by Lemma~\ref{G-cr-L-cr}. 

We see that $P_\alpha$ is a proper parabolic subgroup of $G$ containing $\widetilde H$. Let $P$ be a proper parabolic subgroup of $G$ containing $\widetilde H$. We show that $P=P_\alpha$. Let $\lambda\in Y_{\bar k}(G)$ such that $P_\lambda=P$.
Then $P$ contains $M$. Since $M$ is $G$-cr, $M$ is contained in some Levi subgroup $L$ of $P$. Since any Levi subgroup $L$ of $P$ can be expressed as $L=C_G(u\cdot \lambda(\bar k^*))$ for some $u\in R_u(P_\lambda)$, we may assume that $\lambda(\bar k^*)$ centralizes $M$. From~\cite[Lem.~7.10]{Bate-separability-TransAMS}, we know that $C_G(M)=G_{3\alpha+2\beta}$. So we can write $\lambda$ as
$
\lambda=g\cdot (3\alpha+2\beta)^{\vee} \textup{ for some } g\in G_{3\alpha+2\beta}.
$
By the Bruhat decomposition, $g$ is in one of the following forms:
\begin{alignat*}{2}
&(1)\; g=(3\alpha+2\beta)^{\vee}(s) \epsilon_{3\alpha+2\beta}(x_1), \\
&(2)\; g=\epsilon_{3\alpha+2\beta}(x_1) n_{3\alpha+2\beta} (3\alpha+2\beta)^{\vee}(s)  \epsilon_{3\alpha+2\beta}(x_2)\\
& \textup{ for some } s\in \bar k^*, x_1, x_2\in \bar k.   
\end{alignat*}
We rule out the second case. Suppose that $g$ is in form $(2)$. Since $\widetilde H<P_\lambda=P_{g\cdot (3\alpha+2\beta)^{\vee}} = g\cdot P_{(3\alpha+2\beta)^{\vee}}=g\cdot P_\alpha$, it is enough to show that $g^{-1}\cdot \widetilde H \not\subset P_{\alpha}$. Let
\begin{equation*}
h:=\epsilon_{2\alpha+\beta}(1)\epsilon_{3\alpha+\beta}(a)\epsilon_{3\alpha+2\beta}(\sqrt a)\epsilon_{\alpha+\beta}(\sqrt a)\epsilon_{-\alpha}(\sqrt a)\in \widetilde H.
\end{equation*}
We show that $g^{-1}\cdot h\notin P_\alpha$. Since $h$ centralizes $U_{3\alpha+2\beta}$ and $(3\alpha+2\beta)^{\vee}(s)\epsilon_{3\alpha+2\beta}(x_2)$ belongs to $P_\alpha$ for any $s\in \bar k^*$, $x_2\in \bar k$, without loss, we assume $g=n_{3\alpha+2\beta}$. We compute
\begin{alignat*}{2}
n_{3\alpha+2\beta}^{-1}\cdot h=& (n_\beta n_\alpha n_\beta n_\alpha n_\beta) \cdot h \\
                    =& \epsilon_{-\alpha-\beta}(1)\epsilon_{-\beta}(a)\epsilon_{-3\alpha-2\beta}(\sqrt a)\epsilon_{-2\alpha-\beta}(\sqrt a)
                          \epsilon_{-\alpha}(\sqrt a)\notin P_{\alpha}.
\end{alignat*}

So $g$ must be in form $(1)$ above. Then $g\in P_\alpha$ and $P_\lambda=P_{\alpha}$. Thus we have shown that $P_\alpha$ is the unique proper parabolic subgroup of $G$ containing $\widetilde H$. Since $H<v(\sqrt a)\cdot P_\alpha$, we have 

\begin{lem}\label{G2unique}
$v(\sqrt a)\cdot P_\alpha$ is the unique proper parabolic parabolic subgroup containing $H$.
\end{lem}

\begin{lem}\label{G2nonkdefined}
$v(\sqrt a)\cdot P_\alpha$ is not $k$-defined.
\end{lem}
\begin{proof}
Suppose that $v(\sqrt a)\cdot P_\alpha$ is $k$-defined. Since $P_\alpha$ is $k$-defined, $v(\sqrt a)\cdot P_\alpha$ is $G(k)$-conjugate to $P_\alpha$ by~\cite[Thm.~20.9]{Borel-AG-book}. So we can write $g v(\sqrt a)\cdot P_\alpha=P_\alpha$ for some $g\in G(k)$. Then $g v(\sqrt a)\in P_\alpha$ since parabolic subgroups are self-normalizing. Thus $g=p v(\sqrt a)^{-1}$ for some $p\in P_\alpha$. So $g$ is a $k$-point of $P_\alpha R_u(P_\alpha^{-})$. By the rational version of the Bruhat decomposition~\cite[Thm.~21.15]{Borel-AG-book}, there exist a unique $p'\in P_\alpha$ and a unique $u'\in R_u(P_\alpha^{-})$ such that $g=p'u'$; moreover $p'$ and $u'$ are $k$-points. This is a contradiction since $v(\sqrt a)^{-1}\notin R_u(P_\alpha^{-})(k)$. 
\end{proof}
Lemmas~\ref{G2unique} and \ref{G2nonkdefined} yield Proposition~\ref{G2ir}.
\end{proof}

\begin{prop}\label{G2nonGcr}
$H$ is not $G$-cr.
\end{prop}
\begin{proof}
Recall that $C_G(M) = G_{3\alpha+2\beta}$. Then $C_G(\widetilde H)<G_{3\alpha+2\beta}$ since $M<\widetilde H$. Using the commutation relations, we see that $U_{3\alpha+2\beta}<C_G(\widetilde H)$. Since $\langle 3\alpha+2\beta, (3\alpha+2\beta)^{\vee}\rangle=2$,  $(3\alpha+2\beta)^{\vee}(s)$ does not commute with $h\in \widetilde H$ for any $s\in \bar k^*\backslash\{1\}$. Then $C_G(\widetilde H)=U_{3\alpha+2\beta}$ since $G_{3\alpha+2\beta}=SL_2$. Thus $C_G(H)=v(\sqrt a)\cdot U_{3\alpha+2\beta}$ which is unipotent. So by~\cite[Prop.~3.1]{Borel-Tits-unipotent-invent}, $C_G(H)$ is not $G$-cr. Then~\cite[Cor.~3.17]{Bate-geometric-Inventione} shows that $H$ is not $G$-cr. 
\end{proof}

By Propositions~\ref{G2ir} and \ref{G2nonGcr} we are done.
\begin{rem}
In the proof of Proposition~\ref{G2ir}, $K$ acts non-separably on $R_u(P_\alpha^{-})$. This non-separable action was essential to make $v(\sqrt a)\cdot K$ $k$-defined; see~\cite[Sec.~7]{Bate-separability-TransAMS} for details. 
\end{rem}

\begin{rem}
Note that
\begin{alignat*}{2}
C_G(H) &= \{ v(\sqrt a)\cdot \epsilon_{3\alpha+2\beta}(x) \mid x\in \bar k \} \\
           &= \{ \epsilon_{3\alpha+2\beta}(x)\epsilon_{3\alpha+\beta}(\sqrt a x)\epsilon_{-\beta}(\sqrt a)\epsilon_{\beta}(\sqrt a x)\epsilon_{-\beta}(\sqrt a) \mid x \in \bar k\}\\
           &=\{ \epsilon_{3\alpha+2\beta}(a^{-1})\cdot (\epsilon_{-\beta}(\sqrt a)\epsilon_{\beta}(\sqrt a x)\epsilon_{-\beta}(\sqrt a) \mid x \in \bar k\}
\end{alignat*}
We can identify $\epsilon_{-\beta}(\sqrt a)\epsilon_\beta(\sqrt a x) \epsilon_{-\beta}(\sqrt a)$ with the product of $2\times 2$ matrices in $L_{\beta}= SL_2$:
\begin{alignat*}{2}
\epsilon_{-\beta}(\sqrt a)\epsilon_\beta(\sqrt a x) \epsilon_{-\beta}(\sqrt a)=&  
                       \left[\begin{array}{cc}
                        1 & 0 \\
                        \sqrt a & 1 \\
                         \end{array}  
                        \right]
                        \left[\begin{array}{cc}
                        1 & \sqrt a x \\
                        0 & 1 \\
                         \end{array}  
                        \right]
                         \left[\begin{array}{cc}
                        1 & 0 \\
                        \sqrt a& 1 \\
                         \end{array}  
                        \right]\\
                 =&   \left[\begin{array}{cc}
                        1+ax & \sqrt a x \\
                        a\sqrt a x& 1+ax \\
                         \end{array}  
                        \right].
\end{alignat*}
Then $C:=\{ \epsilon_{-\beta}(\sqrt a)\epsilon_\beta(\sqrt a x) \epsilon_{-\beta}(\sqrt a) \mid x\in \bar k\}$ is $L_\beta$-ir over $k$ since $C$ contains a $k$-anisotropic unipotent element $ \left[\begin{array}{cc}
                        1+a & \sqrt a \\
                        a \sqrt a & 1+a \\
                         \end{array}  
                        \right]$. Thus $C$ is $G$-cr over $k$ by Lemma~\ref{G-cr-L-cr}.
Since $C_G(H)$ is $G(k)$-conjugate to $C$, it is $G$-cr over $k$. Note that this agrees with Proposition~\ref{TCC-centralizer}.

\end{rem}

\subsection{The $E_6$ example}
Let $G/k$ be a simple algebraic group of type $E_6$. By Proposition~\ref{isogeny}, we may assume $G$ is simply-connected. Fix a maximal $k$-split torus $T$ of $G$ and a $k$-Borel subgroup $B$ of $G$ containing $T$.
Let $\Sigma = \{ \alpha,\beta,\gamma,\delta,\epsilon, \sigma \}$ be the set of simple roots of $G$ corresponding to $B$ and $T$. The next figure defines how each simple root of $G$ corresponds to each node in the Dynkin diagram of $E_6$. 
\begin{figure}[h]
                \centering
                \scalebox{0.7}{
                \begin{tikzpicture}
                \draw (1,0) to (2,0);
                \draw (2,0) to (3,0);
                \draw (3,0) to (4,0);
                \draw (4,0) to (5,0);
                \draw (3,0) to (3,1);
                \fill (1,0) circle (1mm);
                \fill (2,0) circle (1mm);
                \fill (3,0) circle (1mm);
                \fill (4,0) circle (1mm);
                \fill (5,0) circle (1mm);
                \fill (3,1) circle (1mm);
                \draw[below] (1,-0.3) node{$\alpha$};
                \draw[below] (2,-0.3) node{$\beta$};
                \draw[below] (3,-0.3) node{$\gamma$};
                \draw[below] (4,-0.3) node{$\delta$};
                \draw[below] (5,-0.3) node{$\epsilon$};
                \draw[right] (3.3,1) node{$\sigma$};
                \end{tikzpicture}}
\end{figure}

We label all positive roots of $G$ in Table~\ref{E6} in Appendix. The labeling for the negative roots follows in the obvious way. Let $L:=L_{\alpha\beta\gamma\delta\epsilon}=\langle T, U_i \mid i\in \{\pm 22,\cdots, \pm 36\}\rangle$. $P:=P_{\alpha\beta\gamma\delta\epsilon}=\langle L, U_i \mid i\in \{1,\cdots, 21\} \rangle$. Then $P$ is a parabolic subgroup of $G$ and $L$ is a Levi subgroup of $P$. Since our argument is similar to that of the $G_2$ example, we just give a sketch. We use the commutation relations~\cite[Lem.~32.5 and Prop.~33.3]{Humphreys-book1} repeatedly. Let 
\begin{alignat*}{2}
q_1:&=n_\alpha n_\beta n_\alpha,\; q_2:=n_\alpha n_\beta n_\gamma n_\beta n_\alpha n_\beta n_\epsilon,\; q_3:=n_\alpha n_\beta n_\alpha n_\delta n_\epsilon n_\delta,\\
q_4:&=n_\alpha n_\beta n_\gamma n_\delta n_\gamma n_\beta n_\alpha n_\gamma n_\delta n_\gamma,\;
q_5:=n_\beta n_\gamma n_\delta n_\epsilon n_\delta n_\gamma n_\beta n_\delta n_\epsilon n_\delta,\\
K:&=\langle q_1, q_2, q_3, q_4, q_5 \rangle<L, \; |K|=72.
\end{alignat*}
We took $q_1, \cdots, q_5$ from Table 1 (case 11) in \cite[Sec.~3]{Uchiyama-Classification-pre}. 
From the Cartan matrix of $E_6$~\cite[Sec.~11.4]{Humphreys-book2}, we see how $n_\alpha, \cdots, n_\epsilon$ act on $\Psi(R_u(P))$. 
Let $\pi:\langle n_\alpha, \cdots, n_\epsilon \rangle \rightarrow \textup{Sym}(\Psi(R_u(P)))\cong S_{21}$ be the corresponding homomorphism. Then
\begin{alignat}{2}
\pi(q_1)&=(2\; 5)(4\; 8)(7\; 11)(10\; 15)(13\; 17)(16\; 19), \nonumber\\
\pi(q_2)&=(1\; 5)(2\; 3)(4\; 17)(6\; 14)(7\; 15)(8\; 11)(9\; 12)(10\; 13)(16\; 20)(18\; 19), \nonumber\\
\pi(q_3)&=(2\; 11)(3\; 9)(4\; 8)(5\; 7)(10\; 19)(12\; 18)(13\; 17)(15\; 16), \nonumber\\
\pi(q_4)&=(1\; 4\; 8)(2\; 12\; 5)(3\; 10\; 15)(7\; 18\; 11)(9\; 16\; 19)(13\; 20\; 17),\nonumber\\
\pi(q_5)&=(1\; 9\; 3)(2\; 7\; 13)(4\; 16\; 10)(5\; 11\; 17)(8\; 19\; 15)(12\; 18\; 20). \label{reflectionE6-1}
\end{alignat}
The orbits of $K$ in $\Psi(R_u(P))$ are 
\begin{equation*}
O_{1}=\{21\},\; O_{2}=\{6, 14\},\; O_{3}=\Psi(R_u(P))\backslash \{6,14,21\}.
\end{equation*}
Let
\begin{alignat*}{2}
M':&=\langle U_i \mid i\in \{\pm 27, \pm 28, \pm 29, \pm 30\} \rangle<L, \\
M:&=\langle K, M' \rangle,\; v(\sqrt a):= \epsilon_{-6}(\sqrt a)\epsilon_{-14}(\sqrt a). 
\end{alignat*}
Note that $v(\sqrt a)$ centralizes $M'$. Define
\begin{equation*}
H:=\langle v(\sqrt a)\cdot M, \epsilon_{2}(1)\rangle= \langle q_1, q_2 \epsilon_{-21}(a), q_3, q_4, q_5, M', \epsilon_{2}(1) \rangle.
\end{equation*}

\begin{prop}\label{E6-1-example}
$H$ is $G$-ir over $k$.
\end{prop}
\begin{proof}
Let 
\begin{equation*}
\widetilde H:=v(\sqrt a)^{-1}\cdot H = \langle M, v(\sqrt a)^{-1}\cdot \epsilon_2(1) \rangle 
            = \langle M, \epsilon_2(1)\epsilon_{-36}(\sqrt a ) \rangle. 
\end{equation*}
Since $U_{-36}<L$, we see that $P$ contains $\widetilde H$. Thus $v(\sqrt a)\cdot P$ contains $H$.
\begin{lem}\label{E6-1-unique}
$v(\sqrt a)\cdot P$ is the unique proper parabolic subgroup of $G$ containing $H$.
\end{lem}
\begin{proof}
It is clear that $M$ is contained in $L$. Note that $[L,L]=SL_6$. We identify $n_\alpha, n_\beta, n_\gamma, n_\delta, n_\epsilon$ with $(1\;2)$, $(2\;3)$, $(3\;4)$, $(4\;5)$, $(5\;6)$ in $S_6$. Then $q_1=(1\;3)$, $ q_2=(1\;4)(2\;3)(5\;6)$,  $q_3=(1\;3)(4\;6)$, $q_4=(1\;5\;3)$ and $q_5=(2\;6\;4)$. Let $T_1:=(\alpha+\beta)^{\vee}(\bar k^*), T_2:=(\beta+\gamma)^{\vee}(\bar k^*), T_3:=(\gamma+\delta)(\bar k^*)$, and $T_4:=(\delta+\epsilon)(\bar k^*)$. Then $T_i$ is a maximal torus of $G_i$ for $i = 27,28,29$ and $30$ respectively. So $\langle T_1, T_2, T_3, T_4 \rangle<M$. 
Now a simple matrix calculation shows that $M$ is $[L,L]$-ir, hence $L$-cr by~\cite[Prop.~2.8]{Bate-commuting-Crelle}. Thus $M$ is $G$-cr by Lemma~\ref{G-cr-L-cr}. 

Let $P_\lambda$ be a proper parabolic subgroup of $G$ containing $\widetilde H$. Then $P_\lambda$ contains $M$. Since $M$ is $G$-cr, without loss we may assume that $\lambda(\bar k^*)$ centralizes $M$. Recall  that by~\cite[Thm.~13.4.2]{Springer-book}, 
$C_{R_u(P)}(M)^\circ\times C_L(M)^{\circ}\times C_{R_u(P^{-})}(M)^{\circ}$ is an open set of $C_G(M)^{\circ}$ where $P^{-}$ is the opposite of $P$ containing $L$. 
\begin{lem}\label{E6-1centralizer}
$C_G(M)^{\circ}=G_{21}$.
\end{lem}
\begin{proof}
First of all, from equations (\ref{reflectionE6-1}), we see that $K$ centralizes $G_{21}$. Using the commutation relations~\cite[Lem.~32.5 and Prop.~33.3]{Humphreys-book1}, $M'$ centralizes $G_{21}$. So $M$ centralizes $G_{21}$. By~\cite[Prop.~8.2.1]{Springer-book}, we write an arbitrary element $u$ of $R_u(P)$ as $u=\prod_{i=1}^{21}\epsilon_i(x_i)$ for some $x_i\in \bar k$. It is not hard to show that if $u\in C_{R_u(P)}(T_1, T_2, T_3, T_4)$, $u$ must be of the form
\begin{equation*}
u=\epsilon_{6}(x_6)\epsilon_{14}(x_{14})\epsilon_{21}(x_{21}) \textup{ for some }x_i\in \bar k. 
\end{equation*}
Then
\begin{alignat*}{2}
q_2\cdot u =& \epsilon_{14}(x_6)\epsilon_{6}(x_{14})\epsilon_{21}(x_{21}) \\
                =& \epsilon_{6}(x_{14})\epsilon_{14}(x_6)\epsilon_{21}(x_{6}x_{14}+x_{21}).
\end{alignat*}
So, for $u\in C_{R_u(P)}(M)$, $x_6=x_{14}=0$. Thus $C_{R_u(P)}(M)=U_{21}$. Likewise $C_{R_u(P^{-})}(M)=U_{-21}$. Note that $C_L(M)<C_L(T_1, T_2, T_3, T_4)$. We find by direct computations that $C_L(T_1, T_2, T_3, T_4)=T$ and $C_T(K)=(\alpha+2\beta+3\gamma+2\delta+\epsilon+2\sigma)^{\vee}(\bar k^*)<G_{21}$. So we are done.
\end{proof}
Now we have $\lambda(\bar k^*)< G_{21}$. Without loss, set $\lambda=g\cdot (\alpha+2\beta+3\gamma+2\delta+\epsilon+2\sigma)^{\vee}$ for some $g\in G_{21}$. By the Bruhat decomposition, $g$ is in one of the following forms:
\begin{alignat*}{2}
&(1)\; g = (\alpha+2\beta+3\gamma+2\delta+\epsilon+2\sigma)^{\vee}(s)\epsilon_{21}(x_1), \\
&(2)\; g = \epsilon_{21}(x_1)n_{21}(\alpha+2\beta+3\gamma+2\delta+\epsilon+2\sigma)^{\vee}(s)\epsilon_{21}(x_2)\\
& \textup{ for some }x_1, x_2\in \bar k, s\in \bar k^{*}. 
\end{alignat*}
By the similar argument to that of the $G_2$ case, if we rule out the second case we are done. Suppose that $g$ is in form $(2)$. Let $h:=\epsilon_2(1)\epsilon_{-36}(\sqrt a )\in \widetilde H$. It is enough to show that $g^{-1}\cdot h \not\subset P_{(\alpha+2\beta+3\gamma+2\delta+\epsilon+2\sigma)^{\vee}}$. Since $h$ centralizes $U_{21}$ and $\epsilon_{21}(x_2)(\alpha+2\beta+3\gamma+2\delta+\epsilon+2\sigma)^{\vee}(s)$ belongs to $P_{(\alpha+2\beta+3\gamma+2\delta+\epsilon+2\sigma)^{\vee}}$ for any $x_2\in \bar k, s\in \bar k^*$, we may assume $g=n_{21}$. We have
\begin{alignat*}{2}
n_{21}=&n_{\epsilon}n_{\sigma}n_{\delta}n_{\epsilon}n_{\gamma}n_{\sigma}n_{\delta}n_{\epsilon}n_{\gamma}
                      n_{\delta}n_{\beta}n_{\gamma}n_{\sigma}n_{\delta}n_{\epsilon}n_{\gamma}n_{\delta}n_{\beta}
                      n_{\gamma}n_{\sigma}n_{\alpha}n_{\beta}
                      n_{\gamma}n_{\sigma}n_{\delta}n_{\epsilon}
                      n_{\gamma}n_{\delta}n_{\beta}n_{\gamma}\\
                      &n_{\sigma}n_{\alpha}n_{\beta}n_{\gamma}
                      n_{\delta}n_{\epsilon} \textup{ (the longest element in the Weyl group of $E_6$). }
\end{alignat*}
A quick calculation shows $n_{21}\cdot U_2 = U_{-2}$ and $n_{21}\cdot U_{-36}=U_{36}$. Then
\begin{equation*}
n_{21}^{-1}\cdot (\epsilon_2(1)\epsilon_{-36}(\sqrt a ))=\epsilon_{-2}(1)\epsilon_{36}(\sqrt a)\not\in P_{(\alpha+2\beta+3\gamma+2\delta+\epsilon+2\sigma)^{\vee}}.
\end{equation*}
So we are done.
\end{proof}
\begin{lem}\label{E6-1-nonkdefined}
$v(\sqrt a)\cdot P$ is not $k$-defined.
\end{lem}
\begin{proof}
This is similar to Lemma~\ref{G2nonkdefined}.
\end{proof}
\end{proof}

\begin{prop}\label{E6nonG-cr}
$H$ is not $G$-cr.
\end{prop}
\begin{proof}
This is similar to Proposition~\ref{G2nonGcr}. From Lemma~\ref{E6-1centralizer}, $C_G(H)^{\circ}<v(\sqrt a)\cdot G_{21}$. Using the commutation relations, $v(\sqrt a)\cdot U_{21}<C_G(H)$. Note that $\langle -2, (\alpha+2\beta+3\gamma+2\delta+\epsilon+2\sigma)^{\vee}\rangle =-1$. So, $(\alpha+2\beta+3\gamma+2\delta+\epsilon+2\sigma)^{\vee}(s)$ does not commute with $h$ for any $s\in \bar k^*\backslash\{1\}$. A similar argument to that of the $G_2$ case shows that $C_G(H)^{\circ}=v(\sqrt a)\cdot U_{21}$ which is unipotent. So by~\cite[Prop.~3.1]{Borel-Tits-unipotent-invent}, $C_G(H)^{\circ}$ is not $G$-cr. Then $C_G(H)$ is not $G$-cr by Proposition~\ref{Normal} since $C_G(H)^{\circ}$ is a normal subgroup of $C_G(H)$. Now~\cite[Cor.~3.17]{Bate-geometric-Inventione} shows that $H$ is not $G$-cr. 
\end{proof}
By Propositions~\ref{E6-1-example} and~\ref{E6nonG-cr}, we are done.

\begin{rem}
Note that $C_G(H)^{\circ}=\{ \epsilon_{21}(a^{-1})\cdot (\epsilon_{-6}(\sqrt a)\epsilon_{6}(\sqrt a x)\epsilon_{-6}(\sqrt a))\mid x\in \bar k \}$ which is $G$-cr over $k$ by the same argument as that of the $G_2$ example.
\end{rem}

\begin{rem}
One can obtain more examples satisfying Theorem~\ref{G-cr over k non G-cr} using nonseparable subgroups in \cite[Sec.~3,4,5]{Uchiyama-Classification-pre} for $G=E_6,E_7$, and $E_8$; see~\cite{Uchiyama-phdthesis}.
\end{rem}

\section{Proof of Theorem~\ref{non-connected G-cr over k non-G-cr}}

Let $\tilde G/k$ be a simple algebraic group of type $A_4$. Let $G:=\tilde G\rtimes \langle \sigma \rangle$ where $\sigma$ is  the non-trivial graph automorphism of $\tilde G$. Fix a maximal $k$-split torus $T$ and $k$-Borel subgroup $B$ of $G$ containing $T$. Define the set of simple roots $\{\alpha, \beta, \gamma, \delta \}$ of $G$ as in the following Dynkin diagram.
\begin{figure}[!h]
                \centering{
                 \scalebox{0.7}{
                \begin{tikzpicture}
                \draw (0,0) to (1,0);
                \draw (1,0) to (2,0);
                \draw (2,0) to (3,0);
                \fill (0,0) circle (1mm);
                \fill (1,0) circle (1mm);
                \fill (2,0) circle (1mm);
                \fill (3,0) circle (1mm);
                \draw[below] (0,-0.3) node{$\alpha$};
                \draw[below] (1,-0.3) node{$\beta$};
                \draw[below] (2,-0.3) node{$\gamma$};
                \draw[below] (3,-0.3) node{$\delta$};
                \end{tikzpicture}}
                }
\end{figure}
Let $\lambda=(\alpha+\beta+\gamma+\delta)^{\vee}$. Then 
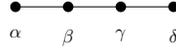
\begin{alignat*}{2}
L_\lambda=&\langle T, G_{\beta}, G_{\gamma}, G_{\beta+\gamma}, \sigma \rangle, \\
P_\lambda=&\langle L, U_i \mid i\in \{\alpha, \delta, \alpha+\beta, \gamma+\delta, \alpha+\beta+\gamma, \beta+\gamma+\delta, \alpha+\beta+\gamma+\delta\} \rangle.
\end{alignat*} 
Let
\begin{alignat*}{2}
K:=&\langle \sigma \rangle, \; v(\sqrt a):=\epsilon_{-\alpha-\beta}(\sqrt a)\epsilon_{-\gamma-\delta}(\sqrt a), \\
M:=&\langle K, G_{\beta+\gamma} \rangle < L_\lambda. 
\end{alignat*}
Note that $v(\sqrt a)$ centralizes $G_{\beta+\gamma}$. Define
\begin{equation*}
H:=\langle v(\sqrt a)\cdot M, \epsilon_{\beta+\gamma+\delta}(1) \rangle. 
\end{equation*}

\begin{prop}
$H$ is $G$-cr over $k$, but not $G$-cr.
\end{prop}
\begin{proof}
We have
\begin{alignat*}{2}
\widetilde H:=v(\sqrt a)^{-1}\cdot H:=& \langle \sigma, v(\sqrt a)^{-1}\cdot G_{\beta+\gamma}, v(\sqrt a)^{-1}\cdot \epsilon_{\beta+\gamma+\delta}(1) \rangle\\
 =&\langle \sigma, G_{\beta+\gamma},  \epsilon_{\beta+\gamma+\delta}(1)\epsilon_{\beta}(\sqrt a )\rangle < P_\lambda. 
\end{alignat*}
Let $M':=\langle \sigma, (\beta+\gamma)^{\vee}(\bar k^*)\rangle <M$. We show that $M'$ is $L_\lambda$-cr. We have
\begin{equation*}
M' < L_{(\beta+\gamma)^{\vee}}(L_\lambda)=C_{L_\lambda}((\beta+\gamma)(\bar k^*))=\langle \sigma, T \rangle.
\end{equation*}
Clearly, $M'$ is $L_{(\beta+\gamma)^{\vee}}(L_\lambda)$-ir. Then, by~\cite[Cor.~3.5 and Sec.~6.3]{Bate-geometric-Inventione}, $M'$ is $L_\lambda$-cr. Now we show that $M$ is $L_\lambda$-cr. Suppose not; then there exists a proper $R$-parabolic subgroup $P_L$ of $L_\lambda$ containing $M$. Let $P_L=P_{\mu}(L_\lambda)$ for some cocharacter $\mu$ of $L_\lambda$. Since $M'$ is $L_\lambda$-cr and $M'$ is contained in $P_\mu(L_\lambda)$, there exists an $R$-Levi subgroup of $P_\mu(L_\lambda)$ containing $M'$. So, without loss, we assume that $\mu(\bar k^*)$ is contained in $C_{L_{\lambda}}(M')=\langle \sigma, T \rangle$. Then $\mu=c\alpha^\vee+d\beta^\vee+d\gamma^\vee+c\delta^\vee$ for some $c,d \in \mathbb{Q}$ since $\sigma$ centralizes $\mu$. Note that $P_\mu(L_\lambda)$ contains $G_{\beta+\gamma}$, so we have $\langle \beta+\gamma, \mu \rangle=0$. Then $\mu=(\alpha+\beta+\gamma+\delta)^\vee$ up to a positive scaler multiple. But then $P_\mu(L_\lambda)=L_\lambda$. This is a contradiction. So, $M$ is $L_\lambda$-cr, and it is $G$-cr by~\cite[Cor.~3.5 and Sec.~6.3]{Bate-geometric-Inventione}.

Let $P_\mu$ be a proper $R$-parabolic subgroup of $G$ containing $\widetilde H$. Since $M$ is $G$-cr, without loss we can assume that $M$ is centralized by $\mu$. It is clear that $G_{\alpha+\beta+\gamma+\delta}<C_G(M)^{\circ}$. Note that $C_G(M)^{\circ}<C_G(\sigma, (\beta+\gamma)^{\vee}(\bar k^*))=G_{\alpha+\beta+\gamma+\delta}$. So, $C_G(M)=G_{\alpha+\beta+\gamma+\delta}$. Thus $\mu(\bar k^*)< G_{\alpha+\beta+\gamma+\delta}$. Set 
\begin{equation*}
\mu:=g\cdot \lambda \textup{ for some }g\in G_{\alpha+\beta+\gamma+\delta}.  
\end{equation*}
Let $n_{\alpha+\beta+\gamma+\delta}:=n_\alpha n_\beta n_\gamma n_\delta n_\gamma n_\beta n_\alpha$. By the Bruhat decomposition, any element $g$ of $G_{\alpha+\beta+\gamma+\delta}$ can be expressed as 
\begin{alignat*}{2}
&(1)\; g=\lambda(s) \epsilon_{\alpha+\beta+\gamma+\delta}(y_1) \textup{ or } \\
&(2)\; g=\epsilon_{\alpha+\beta+\gamma+\delta}(y_1) n_{\alpha+\beta+\gamma+\delta} \lambda(s) \epsilon_{\alpha+\beta+\gamma+\delta}(y_2)  \textup{ for some }s\in \bar k^*, y_1, y_2 \in \bar k.
\end{alignat*}
We rule out the second case. Suppose $g$ is in form (2). Since $\widetilde H\leq P_\mu=P_{g\cdot \lambda}$, it is enough to show that $g^{-1}\cdot \widetilde H\not\subset P_{\lambda}$. Let
\begin{equation*}
h:= \epsilon_{\beta+\gamma+\delta}(1)\epsilon_{\beta}(\sqrt a )
\in \widetilde H.
\end{equation*}
Since $h$ centralizes $U_{\alpha+\beta+\gamma+\delta}$ and  $\epsilon_{\alpha+\beta+\gamma+\delta}(y_2)\lambda(s)$ belongs to $P_\lambda$ for any $y_2\in \bar k, s\in \bar k^*$, without loss, we assume $g=n_{\alpha+\beta+\gamma+\delta}$. Then
\begin{equation*}
g^{-1}\cdot h=\epsilon_{-\alpha}(1)\epsilon_{\beta}(\sqrt a)\not\in P_{\lambda}.
\end{equation*}
Thus $g$ is in form (1) and $g\in P_{\lambda}$, so $P_{\lambda}$ is the unique proper $R$-parabolic subgroup of $G$ containing $\widetilde H$. A similar argument to the $G_2$ and the $E_6$ cases shows that $v(\sqrt a)\cdot P_\lambda$ is not $k$-defined. Thus $H$ is $G$-ir over $k$. 

We find by a direct computation that $C_G(H)^{\circ}=v(\sqrt a)\cdot U_{\alpha+\beta+\gamma+\delta}$, which is unipotent. Then~\cite[Prop.~3.1]{Borel-Tits-unipotent-invent} yields that $C_G(H)^{\circ}$ is not $G^{\circ}$-cr. Thus $C_G(H)^{\circ}$ is not $G$-cr by~\cite[Lem.~6.12]{Bate-geometric-Inventione}. Suppose that $H$ is $G$-cr. Then $C_G(H)$ is $G$-cr by~\cite[Thm.~3.14 and Sec.~6.3]{Bate-geometric-Inventione}. But $C_G(H)^{\circ}$ is a normal subgroup of $C_G(H)^{\circ}$, so $C_G(H)^{\circ}$ is $G$-cr by~\cite[Ex.~5.20]{Bate-uniform-TransAMS}. This is a contradiction.
\end{proof}

\begin{rem}
In the proof of Theorem~\ref{non-connected G-cr over k non-G-cr}, $K$ acts non-separably on $R_u(P_\lambda^{-})$. 
\end{rem}

\section{Related results}
The following was shown in~\cite[Sec.~7]{Bate-separability-TransAMS},~\cite[Sec.~4]{Uchiyama-Separability-JAlgebra},~\cite[Thm.~1.8]{Uchiyama-Classification-pre}. The key to the construction in the proofs was again non-separability.
\begin{thm}
Let $k=k_s$ be a nonperfect field of characteristic $2$. Let $G/k$ be a simple algebraic group of type $E_n$ (or $G_2$). Then there exists a $k$-subgroup $H$ of $G$ such that $H$ is $G$-cr, but not $G$-cr over $k$.
\end{thm}
Note that this is the opposite direction of Theorem~\ref{G-cr over k non G-cr}. We now show the following. The point is that if we allow $G$ to be non-connected, computations become much simpler than the connected cases. 

\begin{thm}\label{non-connected G-cr non-G-cr over k}
Let $k=k_s$ be a nonperfect field of characteristic $2$. Let $\tilde G/k$ be a simple algebraic group of type $A_2$. Let $G:=\tilde G\rtimes \langle \sigma \rangle$ where $\sigma$ is the non-trivial graph automorphism of $\tilde G$. Then there exists a $k$-subgroup $H$ of $G$ such that $H$ is $G$-cr but not $G$-cr over $k$.
\end{thm}
\begin{proof}
Let $G$ be as in the hypotheses. Fix a maximal $k$-split torus $T$ of $G$ and a $k$-Borel subgroup containing $T$. Let $\{\alpha, \beta\}$ be the set of simple roots of $G$ corresponding to $T$ and $B$. Let $\lambda:=(\alpha+\beta)^{\vee}$. Then $P_\lambda=\langle B, \sigma \rangle$ is a $k$-defined $R$-parabolic subgroup of $G$ and  
$L_\lambda=\langle T, \sigma \rangle= C_{G}((\alpha+\beta)^{\vee}(\bar k^{*}))$ is a $k$-defined $R$-Levi subgroup of $P_\lambda$. Let $K:=\langle \sigma \rangle$ and $v(\sqrt a):=\epsilon_\alpha(\sqrt a)\epsilon_\beta(\sqrt a)$. Define
\begin{equation*}
H:=v(\sqrt a)\cdot K=\langle \sigma\epsilon_{\alpha+\beta}(a)\rangle.
\end{equation*}

First, we prove that $H$ is $G$-cr. It is enough to show that $K$ is $G$-cr since $H$ is $G$-conjugate to $K$. It is clear that $K$ is contained in $L_\lambda$ and $K$ is $L_\lambda$-ir, so $K$ is $G$-cr by \cite[Sec.~6.3]{Bate-geometric-Inventione}.

Now we show that $H$ is not $G$-cr over $k$. Suppose the contrary. It is clear that $P_\lambda$ contains $H$. Then there exists a $k$-defined $R$-Levi subgroup $L$ of $P_\lambda$ containing $H$. By~\cite[Lem.~2.5(\rmnum{3})]{Bate-uniform-TransAMS}, there exists $u\in R_u(P_\lambda)(k)$ such that $L=u\cdot L_\lambda$. Then $u^{-1}\cdot H\leq L_\lambda$. It is obvious that $v(\sqrt a)^{-1}\cdot H \leq L_\lambda$. Let $\pi_\lambda: P_\lambda \rightarrow L_\lambda$ be the canonical projection. For any $s\in H$, we have
\begin{equation*}
u^{-1}\cdot s = \pi_\lambda(u^{-1}\cdot s)=\pi_\lambda(s)=\pi_\lambda(v(\sqrt a)^{-1}\cdot s)=v(\sqrt a)^{-1}\cdot s.
\end{equation*}
So, $u=v(\sqrt a)z $ for some $z\in C_{R_u(P_\lambda)}(K)(k)$.
We compute
$
C_{R_u(P_\lambda)}(K)=U_{\alpha+\beta}.
$
So, 
$
u=v(\sqrt a)\epsilon_{\alpha+\beta}(x)
 =\epsilon_\alpha(\sqrt a)\epsilon_\beta(\sqrt a)\epsilon_{\alpha+\beta}(x) \textup{ for some } x \in k.
$
This is a contradiction since $u$ is a $k$-point. Thus $H$ is not $G$-cr over $k$.
\end{proof}

To finish the paper, we consider another application of non-separability for non-connected $G$ with a slightly different flavor. In~\cite[Sec.~7]{Bate-separability-TransAMS},~\cite[Sec.~3]{Uchiyama-Separability-JAlgebra},~\cite[Thm.~1.2]{Uchiyama-Classification-pre}, it was shown that
\begin{thm}
Let $k$ be an algebraically closed field of characteristic $2$. Let $G/k$ be a simple algebraic group of type $E_n$ (or $G_2$). Then there exists a pair of reductive subgroups $H<M$ of $G$ such that $H$ is $G$-cr but not $M$-cr.
\end{thm}
The following is much easier to prove than the connected cases.
\begin{thm}\label{G-crM-cr}
Let $k$ be an algebraically closed field of characteristic $2$. Let $\tilde G/k$ be a simple algebraic group of type $A_2$. Let $G:=\tilde G\rtimes \langle \sigma \rangle$ where $\sigma$ is the non-trivial graph automorphism of $\tilde G$. Then there exists a pair of reductive subgroups $H<M$ of $G$ such that $H$ is $G$-cr but not $M$-cr.
\end{thm}
\begin{proof}
Use the same $H=\langle \sigma\epsilon_{\alpha+\beta}(a)\rangle$ as in the proof of Theorem~\ref{non-connected G-cr non-G-cr over k}. Then $H$ is $G$-cr. We show that $H$ is not $M$-cr. Set $M:=\langle \sigma, G_{\alpha+\beta}\rangle$. Let $\lambda:=(\alpha+\beta)^{\vee}$. Then the image of $\sigma\epsilon_{\alpha+\beta}(a)$ under the canonical projection $\pi_\lambda:P_\lambda\rightarrow L_\lambda$ is $\sigma$. By~\cite[Lem.~2.17, Thm. 3.1]{Bate-geometric-Inventione} and~\cite[Thm.~3.3]{Bate-uniform-TransAMS}, it is enough to show that $\sigma$ is not $R_u((P_\lambda)(M))$-conjugate to $\sigma\epsilon_{\alpha+\beta}(a)$.
This is easy since $R_u((P_\lambda)(M))=U_{\alpha+\beta}$ which is centralized by $\sigma$. We are done.
\end{proof}

\section*{Acknowledgements}
This research was supported by Marsden Grant UOA1021. The author would like to thank Benjamin Martin, Philippe Gille, and Brian Conrad for helpful discussions. 
\clearpage
\section*{Appendix}

\begin{table}[!h]
\begin{center}
\scalebox{0.7}{
\begin{tabular}{cccc}
\rootsE{1}{1}{0}{0}{0}{0}{0}&\rootsE{2}{1}{0}{0}{1}{0}{0}&\rootsE{3}{1}{0}{1}{1}{0}{0}&\rootsE{4}{1}{0}{0}{1}{1}{0}\\
&&&\\
\rootsE{5}{1}{1}{1}{1}{0}{0}&\rootsE{6}{1}{0}{1}{1}{1}{0}&\rootsE{7}{1}{0}{0}{1}{1}{1}&\rootsE{8}{1}{1}{1}{1}{1}{0}\\
&&&\\
\rootsE{9}{1}{0}{1}{1}{1}{1}&\rootsE{10}{1}{0}{1}{2}{1}{0}&\rootsE{11}{1}{1}{1}{1}{1}{1}&\rootsE{12}{1}{1}{1}{2}{1}{0}\\
&&&\\
\rootsE{13}{1}{0}{1}{2}{1}{1}&\rootsE{14}{1}{1}{1}{2}{1}{1}&\rootsE{15}{1}{1}{2}{2}{1}{0}&\rootsE{16}{1}{0}{1}{2}{2}{1}\\
&&&\\
\rootsE{17}{1}{1}{2}{2}{1}{1}&\rootsE{18}{1}{1}{1}{2}{2}{1}&\rootsE{19}{1}{1}{2}{2}{2}{1}&\rootsE{20}{1}{1}{2}{3}{2}{1}\\
&&&\\
\rootsE{21}{2}{1}{2}{3}{2}{1}&\rootsE{22}{0}{1}{0}{0}{0}{0}&\rootsE{23}{0}{0}{1}{0}{0}{0}&\rootsE{24}{0}{0}{0}{1}{0}{0}\\
&&&\\
\rootsE{25}{0}{0}{0}{0}{1}{0}&\rootsE{26}{0}{0}{0}{0}{0}{1}&\rootsE{27}{0}{1}{1}{0}{0}{0}&\rootsE{28}{0}{0}{1}{1}{0}{0}\\
&&&\\
\rootsE{29}{0}{0}{0}{1}{1}{0}&\rootsE{30}{0}{0}{0}{0}{1}{1}&\rootsE{31}{0}{1}{1}{1}{0}{0}&\rootsE{32}{0}{0}{1}{1}{1}{0}\\
&&&\\
\rootsE{33}{0}{0}{0}{1}{1}{1}&\rootsE{34}{0}{1}{1}{1}{1}{0}&\rootsE{35}{0}{0}{1}{1}{1}{1}&\rootsE{36}{0}{1}{1}{1}{1}{1}\\
\end{tabular}
}
\end{center}
\caption{The set of positive roots of $E_6$}\label{E6}
\end{table}

\bibliography{mybib}

\end{document}